\documentclass{amsart}
\usepackage{amsmath}
\usepackage{amssymb}
\usepackage{multicol}
\usepackage{amsfonts}
\usepackage{amsthm}
\usepackage{enumerate}
\usepackage{tagging}
\usepackage[margin=1.5in]{geometry}
\usepackage{mathtools}
\usepackage{graphicx}
\usepackage{mathrsfs}
\usepackage{hyperref}
\usepackage[table,xcdraw]{xcolor}
\usepackage[T1]{fontenc}
\usepackage[utf8]{inputenc}
\usepackage{float}
\usepackage{tikz-cd}

\usepackage{enumitem}

\usepackage{blkarray,booktabs}

\usepackage{dsfont}

\usepackage{siunitx}
\usepackage[e]{esvect}
\usepackage{bbm}
\usepackage{bm}
\usepackage[scaled]{beramono}
\usepackage{cleveref}
\usepackage{fancyhdr}

\hypersetup{
    colorlinks,
    linkcolor={red!50!black},
    citecolor={blue!50!black},
    urlcolor={blue!80!black}
}

\def\bal#1\nal{\begin{align*}#1\end{align*}}
\def\ball#1\nall{\begin{align}#1\end{align}}
\def\lbal#1\lnal{\begin{flalign*}#1\end{flalign*}}

\def\BAL#1\NAL{\[
\resizebox{\textwidth}{!}{$#1$}
\]}

\makeatletter
\renewcommand*\env@matrix[1][\arraystretch]{%
  \edef\arraystretch{#1}%
  \hskip -\arraycolsep
  \let\@ifnextchar\new@ifnextchar
  \array{*\c@MaxMatrixCols c}}
\makeatother

\newcommand{\rkE}{\operatorname{rk} \hspace{.02cm} E_\alpha(\Q(i))}

\newcommand{\vp}{\varphi}
\newcommand{\wvp}{\widehat{\vp}}

\newcommand{\Z}{\mathbb{Z}}
\newcommand{\Q}{\mathbb{Q}}

\newcommand{\goto}{\rightarrow}

\newcommand{\modd}{\text{ (mod }}

\theoremstyle{definition}

\newtheorem{thm}{Theorem}[section]

\newtheorem{lem}[thm]{Lemma}
\newtheorem{cor}[thm]{Corollary}

\numberwithin{equation}{section}

\newcommand{\p}{\mathfrak{p}}

\newcommand{\q}{\mathfrak{q}}

\DeclareFontFamily{U}{wncy}{}
    \DeclareFontShape{U}{wncy}{m}{n}{<->wncyr10}{}
    \DeclareSymbolFont{mcy}{U}{wncy}{m}{n}
    \DeclareMathSymbol{\Sh}{\mathord}{mcy}{"58}

\makeatletter
\def\legg@dash#1#2{\hb@xt@#1{%
  \kern-#2\p@
  \cleaders\hbox{\kern.5\p@
    \vrule\@height.2\p@\@depth.2\p@\@width\p@
    \kern.5\p@}\hfil
  \kern-#2\p@
  }}
\def\@legg#1#2#3#4#5{\mathopen{}\left[
  \sbox\z@{$\genfrac{}{}{0pt}{#1}{#3#4}{#3#5}$}%
  \dimen@=\wd\z@
  \kern-\p@\vcenter{\box0}\kern-\dimen@\vcenter{\legg@dash\dimen@{#2}}\kern-\p@
  \right]_4\mathclose{}}
\newcommand\legg[2]{\mathchoice
  {\@legg{0}{1}{}{#1}{#2}}
  {\@legg{1}{.5}{\vphantom{1}}{#1}{#2}}
  {\@legg{2}{0}{\vphantom{1}}{#1}{#2}}
  {\@legg{3}{0}{\vphantom{1}}{#1}{#2}}
}
\def\dlegg{\@legg{0}{1}{}}
\def\tlegg{\@legg{1}{0.5}{\vphantom{1}}}
\makeatother

\makeatletter
\def\leggg@dash#1#2{\hb@xt@#1{%
  \kern-#2\p@
  \cleaders\hbox{\kern.5\p@
    \vrule\@height.2\p@\@depth.2\p@\@width\p@
    \kern.5\p@}\hfil
  \kern-#2\p@
  }}
\def\@leggg#1#2#3#4#5{\mathopen{}\left[
  \sbox\z@{$\genfrac{}{}{0pt}{#1}{#3#4}{#3#5}$}%
  \dimen@=\wd\z@
  \kern-\p@\vcenter{\box0}\kern-\dimen@\vcenter{\leggg@dash\dimen@{#2}}\kern-\p@
  \right]_4^2\mathclose{}}
\newcommand\leggg[2]{\mathchoice
  {\@leggg{0}{1}{}{#1}{#2}}
  {\@leggg{1}{.5}{\vphantom{1}}{#1}{#2}}
  {\@leggg{2}{0}{\vphantom{1}}{#1}{#2}}
  {\@leggg{3}{0}{\vphantom{1}}{#1}{#2}}
}
\def\dleggg{\@leggg{0}{1}{}}
\def\tleggg{\@leggg{1}{0.5}{\vphantom{1}}}
\makeatother

\newcommand{\F}{\mathbb{F}}

\newcommand{\LEGGG}[2]{\genfrac{[}{]}{}{}{\mspace{3mu} #1 \mspace{3mu}}{\mspace{3mu} #2 \mspace{3mu}}_2}

\newcommand{\vecd}{\vec{\mathbbm{1}}^{(d)}}

\begin{document}

\title{Infinitely many elliptic curves over $\Q(i)$ with rank 2 and $j$-invariant 1728}
\author{Ben Savoie}

\begin{abstract}
We prove that there exist infinitely many elliptic curves over $\mathbb{Q}(i)$ with $j$-invariant $1728$ and rank exactly $2$ which are not obtained by base change from $\mathbb{Q}$. The rank of each such curve is determined via 2-isogeny descent, and the existence of infinitely many such curves follows from Tao's constellation theorem for Gaussian primes.
\end{abstract}

\maketitle

\markboth{\textnormal{\footnotesize Ben Savoie}}{\textnormal{\footnotesize Infinitely many elliptic curves over $\Q(i)$ with rank 2 and $j$-invariant 1728}}
\section{Introduction}For an elliptic curve $E$ over a number field $K$, the abelian group of rational points $E(K)$ is finitely generated. For any number field $K$, Merel showed that only finitely many torsion subgroups of $E(K)$ can occur \cite{merel1996bornes}, and when $[K: \Q] \leq 2$, the possible torsion groups have been classified by Mazur \cite{mazur1977modular} and Kamienny–Kenku–Momose (\cite{kenku1988torsion}, \cite{kamienny1992torsion}). However, the \textit{rank} of $E$ (i.e., the rank of $E(K)$) is far more elusive. 

Although it is known that infinitely many elliptic curves over $\Q$ have rank \textit{at least} 2, it is expected that such curves comprise a density zero subset of all elliptic curves over $\Q$. This naturally raises the question of whether there are infinitely many elliptic curves over $\Q$ with rank \textit{exactly} 2, which was recently resolved by Zywina \cite{zywina2025there}. Zywina used 2-descent to show that every elliptic curve in a certain family has rank 2, and then established the existence of infinitely many such curves by applying the \textit{polynomial Szemerédi theorem for primes} of Tao and Ziegler. Moreover, he showed that these curves all have distinct $j$-invariants; that is, no two of them are isomorphic over $\overline{\Q}$.

Zywina’s work builds on a broader framework that combines tools from additive combinatorics with descent methods. This approach was first implemented by Koymans and Pagano in their recent article \cite{koymans2024hilbert} on Hilbert's Tenth Problem, where they combined 2-descent with a Green–Tao type theorem for number fields due to Kai \cite{kai2023linear} to construct elliptic curves whose rank does not grow in certain quadratic extensions. Specifically, they proved that if $K$ is a number field with sufficiently many real embeddings, then there exists an elliptic curve $E/K$ such that $\operatorname{rk} \hspace{.02cm} E(K(i)) = \operatorname{rk} \hspace{.02cm} E(K)> 0$. 

In this paper, we restrict our attention to elliptic curves with $j$-invariant 1728. It is not difficult to show that there are infinitely many such curves over $\Q(i)$ with rank exactly 2. Indeed, if $E/\Q$ has $j(E) = 1728$, then $\mathrm{rk}\,E(\Q(i)) = 2\,\mathrm{rk}\,E(\Q)$ by \cite[10.16]{Sil}, since in this case $E$ is isomorphic to its quadratic twist by $-1$. Thus, it suffices to construct infinitely many elliptic curves over $\Q$ with $j$-invariant 1728 and rank exactly 1. This was achieved by Monsky \cite{monsky1992three}, who showed that the curve $y^2 = x^3 + px$ has rank exactly 1 whenever $p$ is a rational prime congruent to $5$ modulo $16$. 

This illustrates that the question of whether there exist infinitely many elliptic curves over a number field $K$ of some fixed high rank becomes less interesting when one permits curves arising by base change from a smaller field. This is further evidenced by \cite[Section 12]{park2019heuristic}, which  shows that the contribution from curves defined over subfields can cause the quantity 
\begin{align}\label{defn-BK}
\max\{ n \geq 0 : \text{there exist infinitely many non-isomorphic } E/K \text{ with rank } n\}
\end{align} 
to grow with $[K:\Q]$ (assuming bounded rank)\footnote{Note that the quantity defined in (\ref{defn-BK}) agrees with the definition of $B_K$ in \cite[Section 12]{park2019heuristic} if we assume that the rank of all $E/K$ is bounded.}. For this reason, we call an elliptic curve $E/K$ \textit{genuinely defined over }$K$ if it is not the base change of an elliptic curve over a proper subfield of $K$. Our main result is the following:

\begin{thm}\label{thm-infinitely-rank2}
There are infinitely many elliptic curves genuinely defined over $\Q(i)$ with $j$-invariant 1728 and rank exactly 2. 
\end{thm}
Elliptic curves over \(\mathbb{Q}(i)\) with \(j\)-invariant 1728 arise as \emph{quartic twists} of the curve \(y^2 = x^3 + x\), and can be expressed in the form \cite[X.5.4]{Sil}
\begin{align}\label{quartic-twist-form}
E_\alpha: \; y^2 = x^3 + \alpha x \quad \text{for some } \alpha \in \mathbb{Q}(i)^\times / (\mathbb{Q}(i)^\times)^4.
\end{align}
We call $E_\alpha$ a \emph{congruent number curve} if \(\alpha\) is a square, in which case there is a unique square-free $\gamma \in \mathbb{Z}[i]$ with $E_\alpha \cong E_{\gamma^2}$. We in fact prove that there are infinitely many congruent number curves genuinely defined over \(\mathbb{Q}(i)\) with rank 2.

\begin{thm}\label{thm-rank2-congruent-number-curve}
Let $\beta \in \Z[i]$ and $k \in \Z$ be such that each of the four Gaussian integers
\[
\beta+ k(1+i), \quad \beta + k i(1+i), \quad \beta -k(1+i), \quad \beta - k i (1+i)
\]
is a Gaussian prime congruent to $-1-6i$ modulo $16.$ If $E$ is the elliptic curve given by
\[
y^2 = x^3 -(\beta^4 +4k^4)^2x,
\]
then $E$ is genuinely defined over $\Q(i)$ and $E(\Q(i)) \cong \Z^2 \oplus (\Z/2\Z)^2$.
\end{thm}

To prove Theorem \ref{thm-infinitely-rank2}, we apply Tao’s \textit{constellation theorem in the Gaussian primes} \cite[Theorem 1.2]{tao2006gaussian} to establish infinitely many pairs $(\beta,k)$ satisfying the hypotheses of Theorem \ref{thm-rank2-congruent-number-curve}. This approach is analogous to Zywina’s use of the polynomial Szemerédi theorem and to Koymans-Pagano's use of Kai's Green-Tao type theorem to achieve infinitude.

Since an elliptic curve $E_\alpha / \Q(i)$ of the form (\ref{quartic-twist-form}) has complex multiplication by $\Z[i]$, the group $E_\alpha (\Q(i))$ is a $\Z[i]$-module. In particular, the rank of $E_\alpha$ is necessarily even. Thus, to prove $E_\alpha$ has rank 2, it suffices to produce a single non-torsion $\Q(i)$-point and show that the rank is at most 2. To this end, we first compute all of the $\Q(i)$-torsion points. We then use descent via a 2-isogeny to bound the rank from above.
\subsection*{Acknowledgements} I am deeply grateful to Anthony Várilly-Alvarado for suggesting this problem and for his generous encouragement throughout this project. I also thank Niven Achenjang for an exciting conversation that inspired me to take the final steps toward completing this work. Finally, I thank Anthony Kling for valuable feedback.

\section{\texorpdfstring{\( \Q(i) \)-torsion points}{Q(i)-torsion points}} 
 Our starting point is the following theorem of Najman, which extends Mazur's theorem to elliptic curves over $\Q(i)$.

\begin{thm}[Theorem 1 (ii), \cite{najman2010torsion}]\label{thm-najman} If $E$ is an elliptic curve over $\Q(i)$, then $E(\Q(i))_{\text{tors}}$ is isomorphic to one of the following 17 groups 
\bal 
&\Z/m\Z \quad \hspace{1.6cm} \text{ with } m \in \{1,\dots, 13\}\backslash \{11\},\\
&\Z/2\Z \oplus \Z/2n\Z \quad \hspace{.27cm} \text{ with } n \in \{1,2,3,4\},\\
&\Z/4\Z \oplus \Z/4\Z.
\nal 
\end{thm}
We first narrow down the possibilities by determining the 2-torsion $\Q(i)$-points.
\begin{lem}\label{lem-2-torsion} Let $E_{\gamma^2}/\mathbb{Q}(i)$ be the elliptic curve $y^2 = x^3 + \gamma^2 x$ for some $\gamma \in \mathbb{Z}[i]$. Then 
\[
E_{\gamma^2}(\Q(i))[2] = \{ O, \ (0,0), \ (\gamma i,0), \ (-\gamma i, 0) \} \cong \Z/2\Z \oplus \Z/2\Z.
\]
\end{lem}
\begin{proof} For all $(x,y) \in E_{\gamma^2}(\Q(i))\backslash \{O\}$, the duplication formula \cite[III.2.3]{Sil} yields
$$
2(x,y) = O \iff x(x+ i \gamma)(x-i \gamma) = 0 \text{ and }y=0.
$$
\end{proof}
Combining Theorem \ref{thm-najman} with Lemma \ref{lem-2-torsion} shows that 
\begin{align}\label{b-square-torsion-possibilities}
E_{\gamma^2}(\Q(i))_{\text{tors}}\cong (\Z/4\Z)^2 \text{ or }\Z/2\Z \oplus \Z/2n\Z \text{ for some }n \in \{1,2,3,4\}.
\end{align}
To determine which of these torsion subgroups occur, we will analyze when points of order \(3\) or \(4\) exist. This is done using the \emph{division polynomials} \(\psi_3\) and \(\psi_4\): any point \((x_0, y_0)\) of order \(i \in \{3, 4\}\) must satisfy \(\psi_i(x_0) = 0\). For the general definition of the division polynomials \(\psi_i\), see \cite[Section 3.2]{washington2008elliptic}; we will state only the specialized forms relevant to our setting.

\begin{thm}\label{thm-torsion} Let $\gamma \in \Z[i]$ be square-free. Then 
\bal 
E_{\gamma^2}(\Q(i))_{\text{tors}} \cong \begin{cases}
    \Z/2\Z \oplus \Z/4\Z &\text{if }\gamma = \pm i,\\
    \Z/2\Z \oplus \Z/2\Z &\text{if }\gamma \neq \pm i.
\end{cases}
\nal 
\end{thm}
\begin{proof} By \cite[Theorem 3.6]{washington2008elliptic}, any point $(x_0, y_0) \in E_{\gamma^2}(\Q(i))$ of order 3 must satisfy 
\bal 
\psi_3(x_0) = 3 (x_0/\gamma)^4 + 6 (x_0/\gamma)^2 - 1 = 0.
\nal 
However, this would imply $(x_0/\gamma)^2 = -1 \pm 2 /\sqrt{3}
$, which is impossible since $\sqrt{3} \not\in \Q(i)$. Thus, there are no points of order 3 in $E_{\gamma^2}(\Q(i))$ for any $\gamma$.

Similarly, any point $(x_0, y_0) \in E_{\gamma^2}(\Q(i))$ of order 4 must have $y_0 \neq 0$ and satisfy 
\bal 
\psi_4(x_0)/y_0= 4 (x_0^2 - \gamma^2)(x_0^4 + 6 \gamma^2 x_0^2 + \gamma^4) = 0.
\nal 
Note that $x^4 + 6 \gamma^2 x^2 + \gamma^4 \neq 0$, since otherwise $(x_0 / \gamma)^2 = -3 \pm 2 \sqrt{2} \in \Q(i)$. Therefore, a point of order 4 must satisfy $\gamma^2 = x_0^2$. In this case, we have 
\bal 
y_0^2 = x_0^3 + \gamma^2 x_0 = i^3 (1+i)^2 x_0^3,
\nal 
which is only possible if $x_0 = i \delta^2$ for some $\delta\in \Z[i]$. But then $\gamma^2 = - \delta^4 = -1$, since $\gamma$ is square-free. In summary, $E_{\gamma^2}$ contains a $\Q(i)$-point of order 4 if and only if $\gamma^2= -1$, in which case \cite[\href{https://www.lmfdb.org/EllipticCurve/2.0.4.1/64.1/CMa/1}{Elliptic curve 64.1-CMa1}]{lmfdb}
$$E_{-1}(\Q(i)) \cong \Z/2\Z \oplus \Z/4\Z.$$ 

The desired result now follows from (\ref{b-square-torsion-possibilities}).
\end{proof} 
\begin{cor}\label{cor-torsion-y-0}Let \(\gamma \in \Z[i] \setminus \{0, \pm i\}\) be square-free. Then every torsion point $(x,y) \in E_{\gamma^2}(\Q(i))_{\text{tors}} \backslash \{ O \}$ satisfies $y=0$.
\end{cor} 
\begin{proof} In this case, $E_{\gamma^2}(\Q(i))_{\text{tors}} = E_{\gamma^2}(\Q(i))[2]$, so the result follows from Lemma \ref{lem-2-torsion}.
\end{proof}

\section{Computing the rank via 2-descent} 
\subsection{\texorpdfstring{Bounding the rank with the $\varphi$-Selmer group}{Bounding the rank with the phi-Selmer group}} Let $K$ be a number field and let $E_\alpha$ be the elliptic curve over $K$ given by $y^2 = x^3 + \alpha x$ for $\alpha \in K$. Then $E_\alpha$ is equipped with a degree 2-isogeny $\varphi: E_\alpha \rightarrow E_{-4 \alpha}$ given by 
\bal 
(x,y)\longmapsto (y^2 / x^2, y(\alpha - x^2) / x^2)
\nal 
with kernel $E_\alpha[\vp] = \{O, (0,0)\}.$ The dual isogeny $E_{-4\alpha}\rightarrow E_{\alpha}$ is given by 
\bal 
(x,y)\longmapsto (y^2 / (4x^2), -y(4\alpha +x^2) / (8x^2))
\nal 
with kernel $E_{-4\alpha}[\wvp] = \{O, (0,0)\}$. Associated to $\vp$, we have the \textit{$\vp$-Selmer group} $\text{S}^{(\varphi)}(E_\alpha / K)$, which fits into the short exact sequence \cite[X.4.2]{Sil}
\begin{align}\label{selmer-ses}
0 \longrightarrow E_{-4\alpha}(K)/\varphi\big(E_\alpha(K)\big) \longrightarrow \text{S}^{(\varphi)}(E_\alpha/K) \longrightarrow \Sh(E_\alpha/K)[\varphi]\longrightarrow 0.
\end{align}
It follows from (\ref{selmer-ses}) (see the proof of \cite[X.6.2b]{Sil}) that the rank of $E_\alpha$ is bounded by the dimensions of the $\varphi$-Selmer group of $E_\alpha$ and the $\wvp$-Selmer group of $E_{-4\alpha}$:
\begin{align}\label{rank-bound}
\rkE \leq {} & \dim_{\mathbb{F}_2} S^{(\varphi)}(E_\alpha/K) 
+ \dim_{\mathbb{F}_2} S^{(\widehat{\varphi})}(E_{-4\alpha}/K) \notag \\
& - \dim_{\mathbb{F}_2} \frac{E_{-4\alpha}(K)[\widehat{\varphi}]}{\varphi(E_\alpha(K)[2])}
- \dim_{\mathbb{F}_2} E_\alpha(K)[2].
\end{align}

If we further assume that $i \in K$ and $\alpha$ is a square, then this bound simplifies as follows.
\begin{lem}\label{lem-rank-bound} 
Let $K/ \Q(i)$ and let $\gamma\in K^\times$ be square-free. Then 
\[
\operatorname{rk} \hspace{.02cm} E_{\gamma^2}(K) 
\leq 2 \dim_{\mathbb{F}_2} \text{S}^{(\varphi)}(E_{\gamma^2}/K) - 2.
\]
\end{lem}
\begin{proof} Since $\widehat{\varphi}\circ \varphi = [2]$, we have 
$$
\varphi(E_{\gamma^2}(K) [2] ) \subset E_{-4\gamma^2}(K) [\widehat{\varphi}] = \{O, (0,0) \}.
$$
Furthermore, we have $(\gamma i, 0) \in E_{\gamma^2}(K)[2]$ and $\varphi(\gamma i , 0) = (0,0),$ so in fact
\bal 
\varphi(E_{\gamma^2}(K)[2]) = E_{-4 \gamma^2}(K)[\widehat{\varphi}].
\nal 
Since $K/ \Q(i)$, we also have $E_{\gamma^2}(K)[2] \cong (\Z/2\Z)^2$ by Lemma \ref{lem-2-torsion}. 
Moreover, note that $S^{(\varphi)}(E_{\gamma^2}/K)$ is isomorphic to $S^{(\wvp)}(E_{-4\gamma^2}/K)$, since the isomorphism
$$E_{-4 \alpha} \overset{\sim}{\longrightarrow} E_\alpha, \quad (x,y) \longmapsto \big(x / (1+i)^2, y/ (1+i)^3 \big)$$
maps $E_{-4\alpha}[\wvp]$ isomorphically onto $E_\alpha[\vp]$.
\end{proof}
\subsection{\texorpdfstring{Computing the $\varphi$-Selmer group}{Computing the phi-Selmer group}} We will compute the $\varphi$-Selmer group using the algorithm introduced in \cite{kling2024computing}. For the elliptic curves considered in Theorem \ref{thm-rank2-congruent-number-curve}, the algorithm simplifies considerably; for convenience, we state the relevant version here. We begin by establishing notation and recalling some basic facts about Gaussian integers.

Every Gaussian integer $\alpha$ admits a unique factorization  
\[
\alpha = i^s (1+i)^t \prod \mathfrak{p}_j^{e_j},
\]  
where $s, t, e_j \geq 0$ and each $\mathfrak{p}_j$ is a \emph{primary} Gaussian prime, i.e., $\mathfrak{p}_j \equiv 1 \bmod (1+i)^3$. For any $\alpha \in \Z[i]$ and Gaussian prime $\mathfrak{p} \nmid \alpha$, the \textit{Gaussian quadratic residue symbol} is defined by  
\[
\LEGGG{\alpha}{\p} \equiv \alpha^{(\text{Nm}(\p) -1) /2} \modd \p) = 
\begin{cases}
    1 & \text{if } \alpha \equiv \gamma^2 \modd \mathfrak{p}) \text{ for some } \gamma \in \Z[i], \\
    -1 & \text{otherwise.}
\end{cases}
\]  
This symbol is multiplicative: if $\p \nmid \alpha, \beta$, then $\LEGGG{\alpha \beta}{\p} = \LEGGG{\alpha}{\p} \LEGGG{\beta}{\p}$. By quartic reciprocity (see \cite[Prop. 6.8]{lemmermeyer2013reciprocity}), we have $\LEGGG{\p}{\q} = \LEGGG{\q}{\p}$ for distinct primary Gaussian primes $\mathfrak{p}, \mathfrak{q}$. 

Our computation of the $\varphi$-Selmer group relies on the following results from \cite{kling2024computing}.

\begin{lem}\cite[Remark 2.2 and Lemma 2.4]{kling2024computing}\label{lem-m-n}
Let \(\alpha \in \Z[i]\) be primary; that is, \(\alpha \equiv 1 \mod (1+i)^3\). Then there exist unique elements \(m_\alpha, n_\alpha \in \Z/4\Z\) such that  
\bal
\alpha &\equiv (1 - 4i)^{m_\alpha} (-1 - 6i)^{n_\alpha} \mod (1+i)^7\\
&\equiv (3+2i)^{\overline{n_\alpha}} \mod 4,
\nal
 where \(\overline{n_\alpha} \in \Z/2\Z\) denotes the reduction of \(n_\alpha\).
Moreover, for all primary \(\alpha, \beta \in \Z[i]\), 
\[
m_{\alpha\beta} = m_\alpha + m_\beta \quad \text{and} \quad n_{\alpha\beta} = n_\alpha + n_\beta.
\]
Finally, if $\p$ is a primary Gaussian prime, then 
\[
\LEGGG{1+i}{\p} = (-1)^{m_\p} \quad \text{and} \quad \LEGGG{i}{\p} = (-1)^{n_\p}.
\]
\end{lem}

\begin{thm}[{\cite[Theorem 5.6, Remark 5.7]{kling2024computing}}]\label{thm-selmer-algorithm} Suppose $\alpha \in \Z[i]$ is either of the form 
\bal 
\alpha = \pm \p_1 \cdots \p_N \quad \text{or}\quad \alpha = - \p_1^2 \cdots \p_N^2,
\nal 
where the $\p_j$ are distinct primary Gaussian primes. Define the matrix $L\in \text{Mat}_{N}(\F_2)$ by
\bal 
L_{i,j} = \begin{cases}
        \log_{-1} \LEGGG{\p_i}{\p_j} &\text{if }i \neq j,\\[1ex]
        \sum\limits_{k\neq i}\log_{-1} \LEGGG{\p_i}{\p_k} &\text{if }i=j,
    \end{cases}
\nal 
where $\log_{-1}(-1) := 1$ and $\log_{-1}(1) := 0.$ For each square-free divisor $d \in \Z[i]$ of $\alpha$, define $\vecd \in \F_2^N$ by $(\vecd)_j= 1$ if and only if  $\p_j \mid d$. Then $
\text{S}^{(\varphi)}(E_\alpha / \Q(i)) \subset \text{S}'$, where
\bal 
S' := \left\{\text{square-free }d \mid \alpha: \big(\vecd \in \text{ker}(L) \text{ and }d = d^+\big) \text{ or }  \big( L(\vecd) = (\overline{n_{\p_j}})_j \text{ and }d = i d^+ \big)\right\},
\nal 
and $d^+$ denotes the unique primary associate of $d$.
\end{thm}
\begin{proof} In the notation of \cite[Theorem 5.6]{kling2024computing}, taking $b = \alpha$, the hypothesis on $\alpha$ ensures that $L = L_b' = L_b^{(1)}$. Since $t_b=0$ (i.e., $(1+i)\nmid \alpha$), it follows from the LSC at $1+i$ condition of \cite[Theorem~5.6]{kling2024computing} that $t_d=0$ as well. The $d \in S'$ are precisely the $d$ which satisfy LSC away from $1+i$ with $t_d = 0$, since $\vec{y}^{\hspace{.05cm}(0,0)} = \vec{0}$ and $\vec{y}^{\hspace{.05cm}(1,0)} = (\overline{n_{\p_j}})_j$.
\end{proof}
We will apply Theorem \ref{thm-selmer-algorithm} to bound the rank of $E$ in the proof of Theorem \ref{thm-rank2-congruent-number-curve}. The main step is the computation of the matrix $L$, which utilizes the following lemma. The particular form of the primes $\mathfrak{p}_j$ allows us to determine the quadratic residue symbols between them.
\begin{lem}\label{lem-constellation-residues} Suppose $\beta \in \Z[i]$ and $k \in \Z$ such that $\p_j := \beta + i^j k (1+i)$ is a Gaussian prime satisfying $\p_j \equiv -1 -6i \modd 16)$ for each $j \in \{1,2,3,4\}$. Then 
\bal 
\LEGGG{\p_1}{\p_4} = \LEGGG{\p_2}{\p_3} = \LEGGG{\p_2}{\p_4} \neq \LEGGG{\p_1}{\p_2} = \LEGGG{\p_1}{\p_3} = \LEGGG{\p_3}{\p_4}.
\nal 
\end{lem}
\begin{proof} The key observation is that $\p_j \equiv -1 -6i \mod (1+i)^7$ implies $m_{\p_j} = 0$ and  $n_{\p_j} = 1$ in the notation of Lemma \ref{lem-m-n}. It then follows from Lemma \ref{lem-m-n} that
\bal 
\LEGGG{1+i}{\p_j} = 1 \quad \text{and}\quad \LEGGG{i}{\p_j} = -1.
\nal 
Then, since $\p_1 - \p_4 = -2k = i(1+i)^2 k$, we have
\bal 
\LEGGG{\p_1}{\p_4} &= \LEGGG{i}{\p_4}\LEGGG{1+i}{\p_4}^2 \LEGGG{k}{\p_4} = -\LEGGG{k}{\p_4},\\
\LEGGG{\p_4}{\p_1} &= \LEGGG{i}{\p_1}^3 \LEGGG{1+i}{\p_1}^2 \LEGGG{k}{\p_1}= -\LEGGG{k}{\p_1}.
\nal 
Thus, $\LEGGG{k}{\p_1} = \LEGGG{k}{\p_4}$ since $\LEGGG{\p_1}{\p_4} = \LEGGG{\p_4}{\p_1}$ by quartic reciprocity. Repeating the same argument for all pairs shows that $\LEGGG{k}{\p_j}$ is independent of $j$. By considering the various differences $\p_i -\p_j$, we obtain
\bal 
\LEGGG{\p_1}{\p_4} &= \LEGGG{\p_2}{\p_4} = \LEGGG{\p_2}{\p_3} = - \LEGGG{k}{\p_j} \\
\LEGGG{\p_1}{\p_2} &= \LEGGG{\p_1}{\p_3} = \LEGGG{\p_3}{\p_4} =\LEGGG{k}{\p_j}.
\nal 
\end{proof}
\begin{proof}[Proof of Theorem \ref{thm-rank2-congruent-number-curve}] First note that $E$ is \textit{not} genuinely defined over $\Q(i)$ if and only if 
\begin{align}\label{alpha-in-Z}
-(\beta^4 + 4 k ^4)^2 = - \p_1^2 \p_2^2 \p_3^2 \p_4^2 \in \Z.
\end{align}
Since each $\p_j$ is a split prime (as $\p_j \equiv 3+2i \pmod 4$ by Lemma \ref{lem-m-n}), we have $\p_1^2 \p_2^2 \p_3^2 \p_4^2 \in \Z$ if and only if the four primes $\p_1, \p_2, \p_3, \p_4$ form conjugate pairs. In particular, (\ref{alpha-in-Z}) implies $\overline{\p_1} \in \{\p_2, \p_3, \p_4\}$. The cases $\overline{\p_1} \in \{\p_3, \p_4\}$ are impossible, as this would force $k=0$, which implies $\p_1 = \p_2 = \p_3 = \p_4$. Nor can we have $\overline{\p_1} = \p_2$, as this would imply $\text{Im}(\beta) = 0$, contradicting $\text{Im}(\beta) \equiv 2 \pmod 8$, which follows from $\p_1 + \p_3 = 2 \beta \equiv -2 - 12 i \modd 16)$. Therefore, (\ref{alpha-in-Z}) does not hold, and so $E$ is genuinely defined over $\Q(i)$.

Next we show that the rank of $E$ is at least 2. Note that 
\bal 
 (x_0,y_0):=(4 \beta^2 k^2 , 2 i \beta k (\beta^4 - 4 k^4)) \in E(\Q(i)).
\nal 
We claim that $y_0 \neq 0$. Indeed, $\beta \neq 0$ since otherwise the $\p_j$ would not be primary; $k \neq 0$ since otherwise the $\p_j$ would not be distinct; and $\beta^4 \neq 4k^4$ since $4$ is not a fourth power in $\Z[i]$. Thus, $(x_0, y_0)$ is non-torsion by Corollary~\ref{cor-torsion-y-0}. Since the rank of $E$ is even, it follows that the rank is at least 2.

To show that the rank is at most 2, we apply Theorem \ref{thm-selmer-algorithm} with $\alpha =- \p_1^2 \ \p_2^2 \ \p_3^2 \  \p_4^2.$  By Lemma \ref{lem-constellation-residues}, the matrix $L$ must be either
\bal 
L = \begin{pmatrix}
    1& 0& 0& 1\\ 0& 0& 1& 1\\ 0& 1& 1& 0\\ 1& 1& 0& 0
\end{pmatrix} \qquad \quad \text{or}\quad \qquad L = \begin{pmatrix}
    0& 1& 1& 0\\ 1& 1& 0& 0\\ 1& 0& 0& 1\\ 0& 0& 1& 1
\end{pmatrix},
\nal 
depending on whether $\LEGGG{k}{\p_j}$ is $1$ or $-1$. In either case, we have 
\bal 
d = d^+ \text{ and }\vecd \in \text{ker}(L) &\iff d \in \{1,\  \p_1 \p_2 \p_3 \p_4 \},\\
d = i d^+ \text{ and } L \cdot \vecd = (\overline{n_{\p_j}})_j = \vec{1} &\iff d \in \{i \p_1 \p_3, \ i \p_2 \p_4\},
\nal 
where $(\overline{n_{\p_j}})_j = \vec{1}$ since $n_{\p_j} = 1$. Therefore, Theorem \ref{thm-selmer-algorithm} yields  
\bal 
\text{S}^{(\varphi)}(E / \Q(i)) \subseteq \{1, \ \p_1 \p_2 \p_3 \p_4, \ i \p_1 \p_3, \ i \p_2 \p_4\} \cong (\Z/2\Z)^2,
\nal 
and so the rank is at most 2 by Lemma \ref{lem-rank-bound}. The torsion is given by Theorem \ref{thm-torsion}.
\end{proof}
\section{\texorpdfstring{Infinitely many rank 2 congruent number curves over $\Q(i)$}{Infinitely many rank 2 congruent number curves over Q(i)}} We now recall Tao’s constellation theorem in the Gaussian primes. 
\begin{thm}\cite[Theorem 1.2]{tao2006gaussian}\label{thm-tao-constellation} Let $\gamma_1, \dots, \gamma_n \in \Z[i]$ be distinct and let $\mathcal{P}$ be a subset of the Gaussian primes of \textit{positive
upper relative Banach density}; i.e., 
\bal 
\limsup_{N\goto\infty} \frac{\# \{x + y i  \in \mathcal{P}: \;  |x|, |y| \leq N\} }{\# \{x + y i \text{ Gaussian prime}: \; |x|, |y| \leq N\}} >0.
\nal 
Then there exist infinitely many pairs $(\beta, k) \in \Z[i]\times \Z$ such that
$$\beta + k \gamma_1, \;\dots,\; \beta + k \gamma_n \in \mathcal{P}.$$
\end{thm}
Note that we have stated a stronger theorem than \cite[Theorem 1.2]{tao2006gaussian}, which is formulated for the case where \(\mathcal{P}\) is the set of Gaussian primes. However, Tao's proof of \cite[Theorem 1.2]{tao2006gaussian} establishes the stronger result that we have stated, as noted in \cite[Section 12]{tao2006gaussian}. Alternatively, Theorem \ref{thm-tao-constellation} is a special case of \cite[Theorem 1.4]{kai2020constellations}.

For each pair $(\beta, k)$ satisfying the hypotheses of Theorem \ref{thm-rank2-congruent-number-curve}, we have shown that the elliptic curve $y^2 = x^3 - (\beta^4 + 4 k^4)^2 x$ over $\Q(i)$ has rank 2. To prove Theorem \ref{thm-infinitely-rank2}, it remains to show that infinitely many such pairs $(\beta, k)$ exist.

\begin{proof}[proof of Theorem \ref{thm-infinitely-rank2}] Let $\mathcal{P}$ denote the set of Gaussian primes congruent to $-1-6i$ modulo 16. Also let $\mathfrak{m}$ denote the modulus $(1+i)^8$ and $C_\mathfrak{m}$ denote the associated ray class group. By \cite[Theorem V.1.7]{milne2011class}, we have a group isomorphism
\bal 
C_\mathfrak{m} \cong (\Z[i]/(1+i)^8)^\times \big/ \{\pm 1,\pm i\}.
\nal 
Thus, each class in $C_\mathfrak{m}$ is uniquely represented by a primary element of $ (\Z[i]/(1+i)^8)^\times$.

By the generalized Dirichlet density theorem \cite[Theorem 2.5]{milne2011class}, the prime ideals whose image in $C_\mathfrak{m}$ correspond to $-1-6i \in (\Z[i]/(1+i)^8)^\times$ 
have \emph{Dirichlet} density 1/32. Equivalently, the set of Gaussian primes 
\bal 
\mathcal{P}' := \big\{ \gamma \in \Z[i]: \gamma \equiv i^k (-1-6i) \modd 16) \text{ for some }k \in \{0,1,2,3\} \big\}
\nal 
has Dirichlet density 1/32. By the Chebotarev density theorem (see, e.g., \cite{kisilevsky2011chebotarev}), $\mathcal{P}'$ also has \emph{natural} density 1/32. As each Gaussian prime in $\mathcal{P}'$ has exactly four associates in $\mathcal{P}'$, and $\mathcal{P}$ consists of the primary associates in $\mathcal{P}'$, it follows that $\mathcal{P}$ has natural density 1/128. Finally, by \cite[Proposition 7.7]{kai2020constellations}, taking the fundamental domain to be the set of primary Gaussian primes, we conclude that $\mathcal{P}$ has positive upper relative Banach density.

To complete the proof, apply Theorem \ref{thm-tao-constellation} with $\mathcal{P}$ and $\gamma_j = i^j (1+i)$ for $j \in \{1,2,3,4\}$.
\end{proof}

\bibliographystyle{amsalpha}
\bibliography{main}

\newcommand{\etalchar}[1]{$^{#1}$}
\renewcommand{\MR}[1]{}\providecommand{\noopsort}[1]{}
\providecommand{\bysame}{\leavevmode\hbox to3em{\hrulefill}\thinspace}
\providecommand{\MR}{\relax\ifhmode\unskip\space\fi MR }
\providecommand{\MRhref}[2]{%
  \href{http://www.ams.org/mathscinet-getitem?mr=#1}{#2}
}
\providecommand{\href}[2]{#2}
\begin{thebibliography}{KMM{\etalchar{+}}20}

\bibitem[Kai23]{kai2023linear}
Wataru Kai, \emph{Linear patterns of prime elements in number fields}, arXiv preprint arXiv:2306.16983 (2023).

\bibitem[Kam92]{kamienny1992torsion}
Sheldon Kamienny, \emph{Torsion points on elliptic curves and q-coefficients of modular forms}, Invent. math \textbf{109} (1992), no.~2, 221--229.

\bibitem[KM88]{kenku1988torsion}
Monsur~A Kenku and Fumiyuki Momose, \emph{Torsion points on elliptic curves defined over quadratic fields}, Nagoya Mathematical Journal \textbf{109} (1988), 125--149.

\bibitem[KMM{\etalchar{+}}20]{kai2020constellations}
Wataru Kai, Masato Mimura, Akihiro Munemasa, Shin-ichiro Seki, and Kiyoto Yoshino, \emph{Constellations in prime elements of number fields}, arXiv preprint arXiv:2012.15669 (2020).

\bibitem[KP24]{koymans2024hilbert}
Peter Koymans and Carlo Pagano, \emph{Hilbert's tenth problem via additive combinatorics}, arXiv preprint arXiv:2412.01768 (2024).

\bibitem[KR11]{kisilevsky2011chebotarev}
Hershy Kisilevsky and Michael~O Rubinstein, \emph{Chebotarev sets}, arXiv preprint arXiv:1112.4945 (2011).

\bibitem[KS24]{kling2024computing}
Anthony Kling and Ben Savoie, \emph{A graph-theoretic approach to computing {S}elmer groups of elliptic curves over $\mathbb{Q}(i)$}, 2024, arXiv preprint arXiv:2410.22714.

\bibitem[Lem13]{lemmermeyer2013reciprocity}
Franz Lemmermeyer, \emph{Reciprocity laws: from {E}uler to {E}isenstein}, Springer Science \& Business Media, 2013.

\bibitem[{LMF}25]{lmfdb}
The {LMFDB Collaboration}, \emph{The {L}-functions and modular forms database}, \url{https://www.lmfdb.org}, 2025, [Online; accessed 13 April 2025].

\bibitem[Maz77]{mazur1977modular}
Barry Mazur, \emph{Modular curves and the {E}isenstein ideal}, Publications Math{\'e}matiques de l'Institut des Hautes {\'E}tudes Scientifiques \textbf{47} (1977), no.~1, 33--186.

\bibitem[Mer96]{merel1996bornes}
Lo{\"\i}c Merel, \emph{Bornes pour la torsion des courbes elliptiques sur les corps de nombres}, Inventiones mathematicae \textbf{124} (1996), no.~1, 437--450.

\bibitem[Mil11]{milne2011class}
James~S Milne, \emph{Class field theory}, 2011.

\bibitem[Mon92]{monsky1992three}
Paul Monsky, \emph{Three constructions of rational points on {$Y^2=X^3 \pm NX$}}, Mathematische Zeitschrift \textbf{209} (1992), no.~1, 445--462.

\bibitem[Naj10]{najman2010torsion}
Filip Najman, \emph{Torsion of elliptic curves over quadratic cyclotomic fields}, arXiv preprint arXiv:1005.0558 (2010).

\bibitem[PPVW19]{park2019heuristic}
Jennifer Park, Bjorn Poonen, John Voight, and Melanie~Matchett Wood, \emph{A heuristic for boundedness of ranks of elliptic curves}, Journal of the European Mathematical Society \textbf{21} (2019), no.~9.

\bibitem[Sil09]{Sil}
Joseph~H Silverman, \emph{The arithmetic of elliptic curves}, vol. 106, Springer, 2009.

\bibitem[Tao06]{tao2006gaussian}
Terence Tao, \emph{The {G}aussian primes contain arbitrarily shaped constellations}, Journal d’Analyse Math{\'e}matique \textbf{99} (2006), 109--176.

\bibitem[Was08]{washington2008elliptic}
Lawrence~C Washington, \emph{Elliptic curves: number theory and cryptography}, Chapman and Hall/CRC, 2008.

\bibitem[Zyw25]{zywina2025there}
David Zywina, \emph{There are infinitely many elliptic curves over the rationals of rank 2}, arXiv preprint arXiv:2502.01957 (2025).

\end{thebibliography}
\end{document}